\documentclass[leqno,12pt,a4paper]{article}
\usepackage[left=2.5cm, top=2.5cm,bottom=2.5cm,right=2.5cm]{geometry}
\usepackage{amsmath,amssymb,amsthm,mathrsfs,calc,graphicx,pdfpages}
\usepackage{tikz}
\usepackage[british]{babel}
\usepackage[abs]{overpic} 

\newtheorem{theorem}{Theorem} 

\newtheorem{lemma}[theorem]{Lemma}
\newtheorem{corollary}[theorem]{Corollary}

{\theoremstyle{definition}

}
{\theoremstyle{theorem}
\newtheorem{remark}[theorem]{Remark}

}

\def\ba{\begin{array}}
\def\ea{\end{array}}
\def\bea{\begin{eqnarray} \label}
\def\eea{\end{eqnarray}}
\def\be{\begin{equation} \label}
\def\ee{\end{equation}}
\def\bit{\begin{itemize}}
\def\eit{\end{itemize}}
\def\ben{\begin{enumerate}}
\def\een{\end{enumerate}}

\def\lan{\langle}
\def\ran{\rangle}

\def\BB{\mathbb{B}}

\def\GG{\mathbb{G}}

\def\NN{\mathbb{N}}

\def\RR{\mathbb{R}}
\def\RRd{\mathbb{R}^{d}}
\def\SS{\mathbb{S}}
\def\SSd{\mathbb{S}^{d-1}}

\def\g{\gamma}

\def\k{\kappa}
\def\l{\lambda}

\def\D{\Delta}

\def\Sig{\Sigma}
\def\O{\Omega}

\def\bE{\mathbf{E}}

\def\bP{\mathbf{P}}

\def\bV{\mathbf{Var}}

\def\cA{\mathcal{A}}

\def\cF{\mathcal{F}}

\def\dint{\textup{d}}

\def\vol{\textup{vol}}

\begin{document}

\title{\bfseries Intrinsic volumes and Gaussian polytopes: \\ the missing piece of the jigsaw}

\author{Christoph Th\"ale\footnotemark[1]\; and Imre B\'ar\'any\footnotemark[2]\;\footnotemark[3]}

\date{} \renewcommand{\thefootnote}{\fnsymbol{footnote}}

\footnotetext[1]{Ruhr University Bochum, Faculty of Mathematics, NA 3/68, D-44780 Bochum, Germany. E-mail: christoph.thaele@rub.de}

\footnotetext[2]{R\'enyi Institute of Mathematics, Hungarian Academy of Sciences, H-1364 Budapest, Hungary. E-mail: barany@renyi.hu}

\footnotetext[2]{University College London, Department of Mathematics, London WC1E 6BT, England.}

\maketitle

\begin{abstract}
The intrinsic volumes of Gaussian polytopes are considered. A lower variance bound for these quantities is proved, showing that, under suitable normalization, the variances converge to strictly positive limits. The implications of this missing piece of the jigsaw in the theory of Gaussian polytopes are discussed.
\end{abstract}

\section{Introduction and results}

Fix a space dimension $d\in\NN$ and denote by $\g_d$ the standard Gaussian measure on $\RR^d$ with density $\varphi_d$ equal to
\begin{equation}\label{eq:DensityNormal}
\varphi_d(x) := (2\pi)^{-{d\over 2}}\,\exp\bigg(-{\|x\|^2\over 2}\bigg)\,,\qquad x\in\RRd\,.
\end{equation}
Given $n\geq d+1$ let $X_1,\ldots,X_n$ be independent random points that are distributed on $\RRd$ according to the probability measure $\g_d$. The random convex hull
$$
K_n:=[X_1,\ldots,X_n]
$$
of these points is a \textit{Gaussian polytope}. These random polytopes are central objects considered in stochastic geometry and are also of importance in convex geometric analysis or coding theory. For example, Gluskin \cite{Gluskin} has used Gaussian polytopes in his analysis of the diameter of the Minkowski compactum and Gaussian polytopes also arise as lower-dimensional shadows of randomly rotated high-dimensional regular simplices as shown by Baryshnikov and Vitale \cite{BaryshnikovVitale}. We refer to the survey article of Reitzner \cite{ReitznerSurvey} for further background information and references.

\medskip

We denote for $\ell\in\{0,\ldots,d\}$ by $V_\ell(K_n)$ the $\ell$th intrinsic volume of $K_n$, that is, 
$$
V_\ell(K_n) = {d\choose \ell}{\k_d\over\k_\ell\k_{d-\ell}}\int_{\GG(d,\ell)}\vol_\ell(K_n|L)\,\nu_\ell(\dint L)\,.
$$
Here, $\GG(d,\ell)$ is the Grassmannian of $\ell$-dimensional linear subspaces of $\RRd$ supplied with the unique Haar probability measure $\nu_\ell$ and $\vol_\ell(K_n|L)$ stands for the $\ell$-dimensional Lebesgue measure of the orthogonal projection $K_n|L$ of $K_n$ onto $L$ measured within the subspace $L$. Moreover, for $j\in\NN$, $\k_j:=\pi^{j/2}\,\Gamma(1+{j\over 2})^{-1}$ denotes the volume of the $j$-dimensional unit ball. The intrinsic volumes are of outstanding importance in convex geometry, since according to a classical theorem of Hadwiger they form a basis of the vector space of all continuous and rigid-motion invariant real-valued valuations on convex sets, cf.\ \cite{SchneiderBook}. For example, $V_d(K_n)=\vol_d(K_n)$ is the volume, $2V_{d-1}(K_n)$ coincides with the surface area and ${2\k_{d-1}\over d\k_d}V_1(K_n)$ corresponds to the mean width of $K_n$.

\medskip

It is well known from the work of Affentranger \cite{Affentranger} that the expectation $\bE[V_\ell(K_n)]$ of $V_\ell(K_n)$ satisfies
$$
\lim_{n\to \infty} (\log n)^{-{\ell \over 2}}\,\bE[V_\ell(K_n)] = {d\choose \ell}{\k_d\over\k_{d-\ell}}\,.
$$
More recently, the asymptotic behaviour of the variance $\bV[V_\ell(K_n)]$ of $V_\ell(K_n)$ has moved into the focus of attention. Using the classical Efron-Stein jackknife inequality Hug and Reitzner \cite{HugReitzner} have obtained a first upper bound of the form $\bV[V_\ell(K_n)] \leq c_d(\log n)^{{\ell-3\over 2}}$ with a constant $c_d\in(0,\infty)$ only depending on the space dimension $d$ (but not on $\ell$). In a remarkable paper of Calka and Yukich \cite{CalkaYukich} the precise variance asymptotic was derived, showing thereby that the upper bound from \cite{HugReitzner} does not have the right order of magnitude. In fact, \cite[Theorem 1.5]{CalkaYukich} says that
\begin{equation}\label{eq:VarAsymptotic}
\lim_{n\to \infty} (\log n)^{{d+3\over2}-\ell}\,\bV[V_\ell(K_n)] = c_{d,\ell}\,,
\end{equation}
with constants $c_{d,\ell}\in[0,\infty)$ only depending on $d$ and on $\ell$. However, using their methods the authors of \cite{CalkaYukich} were not able to exclude the possibility that $c_{d,\ell}=0$. The aim of the present paper is to fill this gap and to show that, in fact, $c_{d,\ell}>0$. This answers a question raised at several places in the literature, see \cite[Section 14]{BaranyVu}, the comment after \cite[Theorem 1.5]{CalkaYukich} or \cite[Remark 3.6]{GroteThaele}. Our result reads as follows:

\begin{theorem}\label{thm:Main}
Let $\ell\in\{1,\ldots,d\}$ and let $K_n$ be a Gaussian polytope. Then there exists a constant $v_{d,\ell}\in(0,\infty)$ only depending on $d$ and on $\ell$ such that
$$
\bV[V_\ell(K_n)] \geq v_{d,\ell}\,(\log n)^{-{d+3\over 2}+\ell}\,,
$$
whenever $n$ is sufficiently large.
\end{theorem}

In particular, Theorem \ref{thm:Main} in conjunction with \eqref{eq:VarAsymptotic} shows that the limit
$$
\lim_{n\to\infty}(\log n)^{{d+3\over 2}-\ell}\,\bV[V_\ell(K_n)] = c_{d,\ell}
$$
exists and takes a strictly positive and finite value.

\begin{remark}\rm 
\begin{itemize}
\item[(i)] Let us first comment on the boundary case $\ell=0$ in Theorem \ref{thm:Main}. Since $V_0(K)={\bf 1}_{\{K\neq\emptyset\}}$ for any convex set $K\subset\RRd$, we have that $V_0(K_n)=1$ with probability one and hence $\bV[V_0(K_n)]=0$.
\item[(ii)]  Since $V_d(K_n)$ is the volume of the Gaussian polytope $K_n$, the case $\ell=d$ is already covered by Theorem 6.1 in \cite{BaranyVu}, which ensures that $v_{d,d}\in(0,\infty)$. Our proof comprises this situation as a special case.
\end{itemize}
\end{remark}

A random polytope model closely related to $K_n$ can be described as follows. For each $n\in\NN$ let $\eta_n$ be a Poisson point process on $\RRd$ with intensity measure $n\g_d$. The convex hull of the points of $\eta_n$ will be denoted by $\Pi_n$ and is called the \textit{Gaussian Poisson polytope}. Following the coupling construction in the proof of \cite[Lemma 7.1]{BaranyVu} one easily sees that expectation and variance asymptotic for $\Pi_n$ are literally the same as for $K_n$. Moreover, the strict positivity of the constants $v_{d,\ell}$ in Theorem \ref{thm:Main} implies that $(\log n)^{{d+3\over 2}-\ell}\bV[V_\ell(\Pi_n)]$ converges to a positive and finite limit. We summarize the missing piece in the proof of this result in the following corollary:

\begin{corollary}\label{cor:PoissonPolytopes}
Let $\ell\in\{1,\ldots,d\}$ and let $\Pi_n$ be the Gaussian Poisson polytope. Then there exists a constant $v_{d,\ell}\in(0,\infty)$ only depending on $d$ and on $\ell$ such that
$$
\bV[V_\ell(\Pi_n)] \geq v_{d,\ell}\,(\log n)^{\ell-{d+3\over 2}}\,,
$$
whenever $n$ is sufficiently large.
\end{corollary}

The result of Theorem \ref{thm:Main} and Corollary \ref{cor:PoissonPolytopes} can be regarded as the missing piece of the jigsaw in the theory of Gaussian polytopes. Let us mention some of the implications that are now immediate:
\begin{itemize}
\item[-] \textit{Central limit theorems}. As explained in \cite{BaranyVu,CalkaYukich}, the positivity of the limiting variance is the only missing piece in the proof of the central limit theorem for the normalized intrinsic volumes of $\Pi_n$. The result follows by the methods developed in \cite{BaranyVu,CalkaSchreiberYukich,CalkaYukich}. Moreover, a de-Poissonization argument similar to that in \cite{BaranyVu} leads to the corresponding result for $K_n$; we omit the details.

\item[-] \textit{Concentration inequalities}. As explained in the recent work \cite{GroteThaele}, the positivity of the limiting variance is the only missing ingredient in the proof of a concentration inequality for $V_\ell(\Pi_n)$. The precise form of such an inequality can now be determined from \cite[Theorem 3.1]{GroteThaele}: For any $\ell\in\{1,\ldots,d\}$ one can find a constant $c\in(0,\infty)$ only depending on $d$ and on $\ell$ such that
\begin{align*}
&\bP\big(|V_\ell(\Pi_n)-\bE[V_\ell(\Pi_n)]|\geq y\, \sqrt{\bV[V_\ell(\Pi_n)]}\,\big)\\
&\qquad \qquad \leq 2\exp\Big(-{1\over 4}\min\Big\{{y^2\over 2^{2d+\ell+5}},c\,(\log n)^{d-1\over 4(2d+\ell+5)}\, y^{1\over 2d+\ell+5}\Big\}\Big)
\end{align*}
for all $y\geq 0$ and sufficiently large $n.$

\item[-] \textit{Marcinkiewicz-Zygmund-type strong laws of large numbers.} The concentration inequality for $V_\ell(\Pi_n)$ mentioned in the previous paragraph can directly be used to derive Marcinkiewicz-Zygmund-type strong laws of large numbers along the lines of the proof of \cite[Theorem 1.3]{GroteThaele}: For any $\ell\in\{1,\ldots,d\}$ and $p > 1-{d+3\over\ell}$ one has that
\begin{align*}
\frac{V_\ell(\Pi_n) - \bE [V_\ell(\Pi_n)]}{(\log n)^{p \frac{\ell}{2}}} \longrightarrow 0 
\end{align*}
with probability one, as $n\to\infty$. Using the monotonicity of intrinsic volumes and a simple coupling argument, one easily verifies that the same result also holds with $\Pi_n$ replaced by $K_n$. In that form, this refines the ordinary strong law of large numbers from \cite[Corollary 1.2]{HugReitzner}, which corresponds to the special case $p=1$.

\item[-] \textit{Moderate deviations}. Moderate deviations for the volume and the face numbers of the Gaussian Poisson polytopes $\Pi_n$ have also been investigated in \cite{GroteThaele}. Again, the only missing piece for the extension of these results to the intrinsic volumes is the positivity of the limiting variances; we omit the details.
\end{itemize}

\begin{remark}\rm 
Let $\l>0$ be an arbitrary real number, let $\eta_\l$ be a Poisson point process on $\RRd$ with intensity measure $\l\g_d$ and denote by $\Pi_\l$ the random convex hull induced by $\eta_\l$. Using the monotonicity of intrinsic volumes and a simple coupling argument, one easily verifies that the result of Corollary \ref{cor:PoissonPolytopes} continues to hold with $\Pi_n$ and $\log n$ replaced by $\Pi_\l$ and $\log\l$, respectively. The same comment applies to the central limit theorem, the concentration inequalities, the Marcinkiewicz-Zygmund-type strong laws of large numbers and to the moderate deviations mentioned above.
\end{remark}

The rest of this paper is structured as follows. In Section \ref{sec:Preparations} we recall the essential steps of a geometric construction from \cite{BaranyVu} and proof some auxiliary results that are needed in the proof of Theorem \ref{thm:Main}. The latter is the content of the final Section \ref{sec:ProofMain}.

\section{Preparations}\label{sec:Preparations}

\subsection{Notation}

The symbols $\|\,\cdot\,\|$ and $\lan\,\cdot\,,\,\cdot\,\ran$ are used for the Euclidean norm and scalar product in $\RRd$, respectively. Moreover, for a set $B\subset\RRd$ we write $[B]$ for the convex hull of $B$. We denote the $d$-dimensional unit ball by $\BB^d:=\{x\in\RRd:\|x\|\leq 1\}$ and write $\SSd:=\{x\in\RRd:\|x\|=1\}$ for the corresponding unit sphere. The normalized surface measure on $\SSd$ is denoted by $\nu_{\SSd}$. Further, for a point $z\in\RRd\setminus\{0\}$ and $\alpha\in[0,\pi/2]$ we write $C(z,\alpha)$ for the closed circular cone whose axis is the halfline $\{tz: t\ge 0\}$ and whose angle is $\alpha$. More formally, if $\sphericalangle(z,x)$ stands for the ordinary angle between $z$ and another point $x\in\RRd$, $C(z,\alpha)$ is given by $C(z,\alpha):=\{x\in\RRd:\sphericalangle (x,z)\le \alpha\}$.

Our underlying probability space is $(\O,\cA,\bP)$ and we implicitly assume that it is rich enough to carry all the random objects we consider in this paper. By $\bE[\,\cdot\,]$ we denote expectation (integration) with respect to $\bP$ and $\bV[\,\cdot\,]$ stands for the variance of the argument random variable. The indicator function of an event $A\in\cA$ is denoted by ${\bf 1}_A$.

For two sequences $(a_n:n\in\NN)$ and $(b_n:n\in\NN)$ we write $a_n\ll b_n$ (or $a_n\gg b_n$) if we can find a constant $c\in(0,\infty)$ not depending on $n$ and an index $n_0\in\NN$ such that $a_n\leq c\,b_n$ (or $a_n\geq c\,b_n$) for all $n\geq n_0$. Finally, $a_n\approx b_n$ means that $a_n\ll b_n\ll a_n$.

In this paper constants are denoted by $c_1,c_2,\ldots$ It is implicitly assumed that these constants are finite and strictly positive, and only depend on the space dimension $d$, unless otherwise stated. 

\subsection{A geometric construction}

In this section we recall a geometric construction as well as some of the results already obtained \cite{BaranyVu} that we use below. We define
$$
r=r(n):=\sqrt{2\log n-\log\log n}\,,\qquad n\in\NN\,,
$$ 
and denote by $\SS(r):=\{x\in\RRd:\|x\|=r\}$ the centred sphere of radius $r$. By $y_1,\ldots,y_m\in\SS(r)$ we denote a maximal system of points such that $\|y_i-y_j\|\geq 2c_1$ for some sufficiently large $c_1$. A simple volume comparison argument provides an estimate for the size of such a set, see \cite[Claim 5.1]{BaranyVu}:

\begin{lemma}\label{lem:EstimateM}
One has that $m\approx (\log n)^{d-1\over 2}$.
\end{lemma}

\begin{figure}[t]
\centering
\begin{overpic} 
[width=0.6\textwidth]{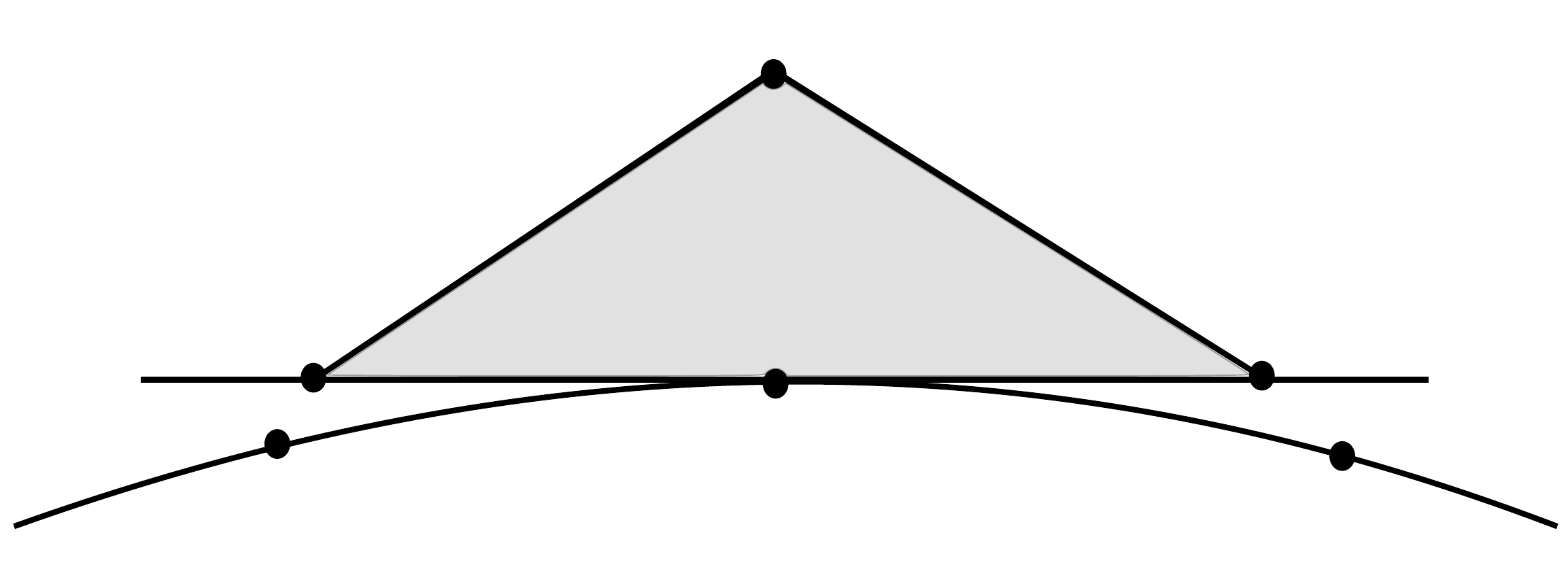} 
\put(0,-5){\scriptsize $\SS(r)$} 
\put(10,30){\scriptsize $H_i$} 
\put(45,10){\scriptsize $y_{i-1}$} 
\put(132,20){\scriptsize $y_i$} 
\put(220,10){\scriptsize $y_{i+1}$} 
\put(47,40){\scriptsize $y_i^1$} 
\put(220,40){\scriptsize $y_i^2$} 
\put(132,95){\scriptsize $y_i^0$} 
\put(132,50){\scriptsize $\D_i$} 
\end{overpic}
\caption{Construction of the simplices $\D_i$.}
\label{fig1}
\end{figure}
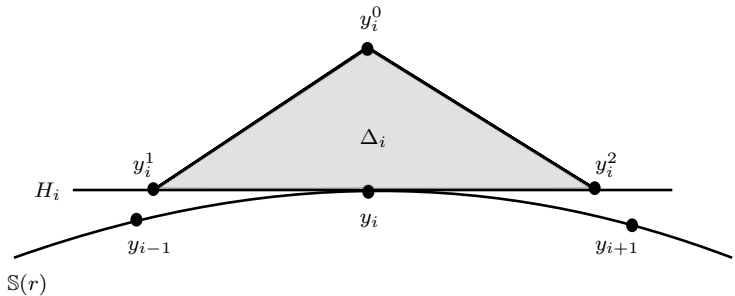

For each $i\in\{1,\ldots,m\}$ define $y_i^0:=(1+r^{-2})y_i$ and notice that $\|y_i-y_i^0\|=r^{-1}$. Let further for $i\in\{1,\ldots,m\}$, $H_i:=\{x\in\RRd:\lan x,y_i\ran = r\}$ be the tangent hyperplane of $\SS(r)$ at $y_i$ and fix a regular simplex in $H_i$ whose vertices $y_i^1,\ldots,y_i^d$ are chosen from the $(d-2)$-dimensional sphere $\SS^{d-2}(y_i,\sqrt{2})$ of radius $\sqrt{2}$ in $H_i$ centred at $y_i$. (Thus $\SS(r)=\SS(0,r)$ but we keep the simpler notation for $\SS(r)$.) The simplex $\D_i:=[y_i^0,y_i^1,\ldots,y_i^d]$ is the convex hull of $y_i^0$ and the points $y_i^1,\ldots,y_i^d\in H_i$, see Figure \ref{fig1}. It is not difficult to estimate the volume $V_d(\D_i)$ and the Gaussian measure $\g_d(\D_i)$ of these simplices, see \cite[Claim 5.2]{BaranyVu}:

\begin{lemma}\label{lem:ContentDeltaI}
For each $i\in\{1,\ldots,m\}$ one has that $V_d(\D_i)\approx(\log n)^{-{1\over 2}}$ and $\g_d(\D_i)\approx n^{-1}$.
\end{lemma}

For each $i\in\{1,\ldots,m\}$ and $j\in\{0,\ldots,d\}$ we let $\D_i^j$ be a homothetic copy of $\D_i$ with $y_i^j$ being the centre of the homothety and the factor being a sufficiently small number $c_2$, that is, $\D_i^j:=y_i^j+c_2(\D_i-y_i^j)$. 

Let $D_i$ be the the cone $D_i:={\rm pos}(\{y_i^j-y_i^0:j\in\{1,\ldots,d\}\})$, where we write ${\rm pos}(\,\cdot\,)$ for the positive hull of the argument set. This is the internal cone at vertex $y_i^0$ of the simplex $\D_i$, which has a simple structure because its base is a $(d-1)$-dimensional regular simplex and the opposite vertex $y_i^0$ is at height $r^{-1}$ over this base exactly above its centre. In particular, one can check easily that 
\begin{equation}\label{eq:Dcones}
C\Big(y_i-y_i^0,\arctan{\sqrt 2 \,r \over d-1}\Big) \subset D_i \subset  C\big(y_i-y_i^0,\arctan\sqrt 2 \,r\big)\,. 
\end{equation}

Since each $\D_i^j$ is only a homothetic copy of $\D_i$ with a scaling factor not depending on $n$, the following holds by construction:

\begin{lemma}\label{lem:ContentDeltaIJ}
For each $i\in\{1,\ldots,m\}$ and $j\in\{0,\ldots,d\}$ one has that $V_d(\D_i^j)\approx(\log n)^{-{1\over 2}}$ and $\g_d(\D_i^j)\approx n^{-1}$.
\end{lemma}

For each $i\in\{1,\ldots,m\}$ and $j\in\{0,\ldots,d\}$ let $z_i^j$ be an arbitrary point in $\D_i^j$ and define the cone $C_i:={\rm pos}(\{z_i^j-z_i^0:j\in\{1,\ldots,d\}\})$. We recall the following fact about these cones from \cite[Lemma 5.4]{BaranyVu}, which ensures a certain independence property used below:

\begin{lemma}\label{lem:ConesCi}
One can choose the constant $c_1$ in the above construction sufficiently large and $c_2$ sufficiently small such that for each $i\in\{1,\ldots,m\}$ the translated cone $z_i^0+C_i$ contains all simplices $\Delta_k$ with $k\in\{1,\ldots,m\}\setminus\{i\}$.
\end{lemma}

Observe further that the simplices $[z_i^0,\ldots,z_i^d]$ and $\D_i$ are very close to each other if the factor of homothety $c_2$ is small enough. So relations (\ref{eq:Dcones}) imply that  
\begin{equation}\label{eq:Ccones}
C_i^1:=C\Big(y_i-z_i^0, \arctan{r \over d-1}\Big)\subset C_i  \subset C(y_i-z_i^0,\arctan 2 r)=:C_i^2\,.
\end{equation}

Next, for $i\in\{1,\ldots,m\}$ and $j\in\{1,\ldots,d\}$ we denote by $H_i^j$ the half-space containing $\D_i^k$ for all $k\in\{0,\ldots,d\}\setminus\{0,j\}$, not containing $\D_i^0$ and $\D_i^j$, and such that the hyperplane bounding $H_i^j$ touches all the simplices $\D_i^0,\ldots,\D_i^d$ except for $\D_i^j$, see Figure \ref{fig:2} (left). We are now in the position to define for each $i\in\{1,\ldots,m\}$ the event $A_i\in\cA$ that precisely one point from the random sample $X_1,\ldots,X_n$ is contained in each simplex of the form $\D_i^j$ and no further point from $X_1,\ldots,X_n$ is contained in $H_i^+\cup H_i^1\cup\ldots\cup H_i^d$, see Figure \ref{fig:2} (right). Here, $H_i^+$ is the half-space bounded by $H_i$ not containing the origin. The following probability estimate is taken from \cite[Lemma 6.2]{BaranyVu}:

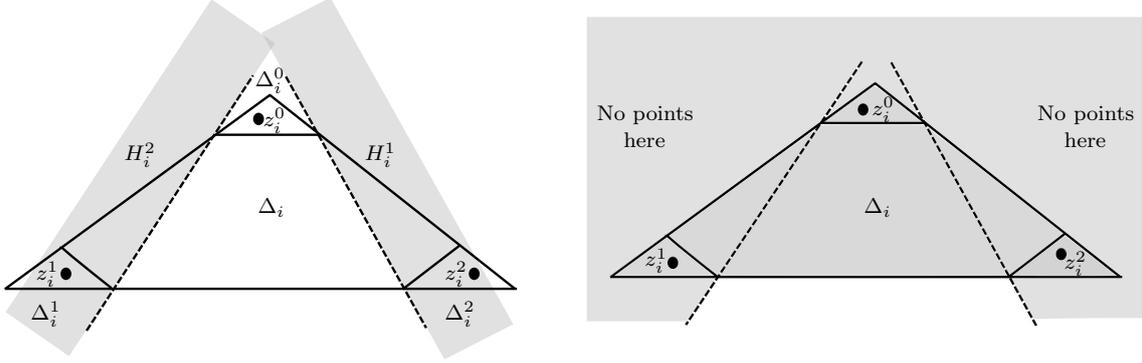
\begin{figure}[t]
\centering
\begin{overpic}
[width=0.49\textwidth]{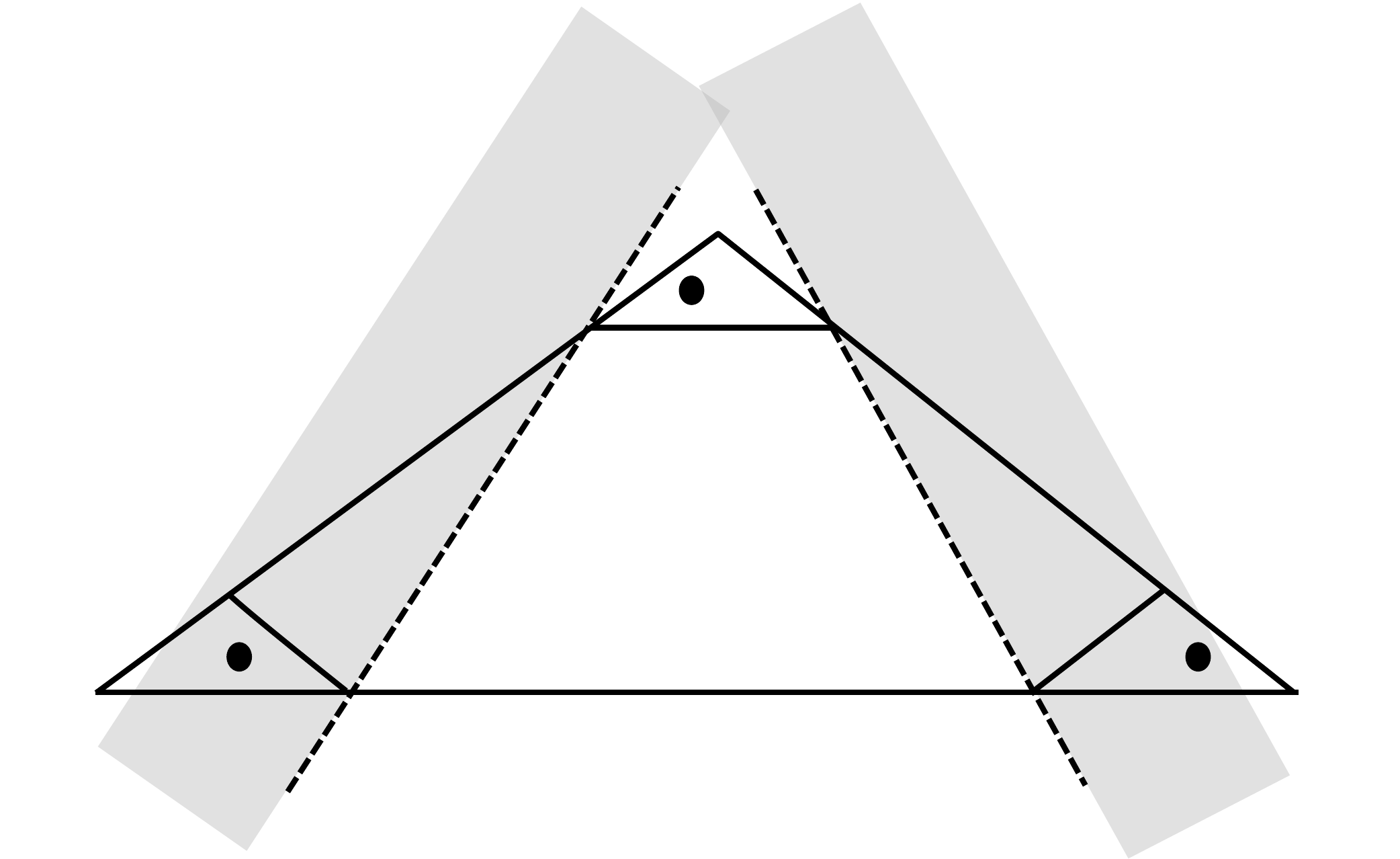} 
\put(110,55){\scriptsize $\D_i$} 
\put(109,103){\scriptsize $\D_i^0$} 
\put(25,15){\scriptsize $\D_i^1$} 
\put(180,15){\scriptsize $\D_i^2$} 
\put(112,88){\scriptsize $z_i^0$} 
\put(27,30){\scriptsize $z_i^1$} 
\put(180,30){\scriptsize $z_i^2$} 
\put(150,75){\scriptsize $H_i^1$} 
\put(60,75){\scriptsize $H_i^2$} 
\end{overpic}
\begin{overpic} 
[width=0.49\textwidth]{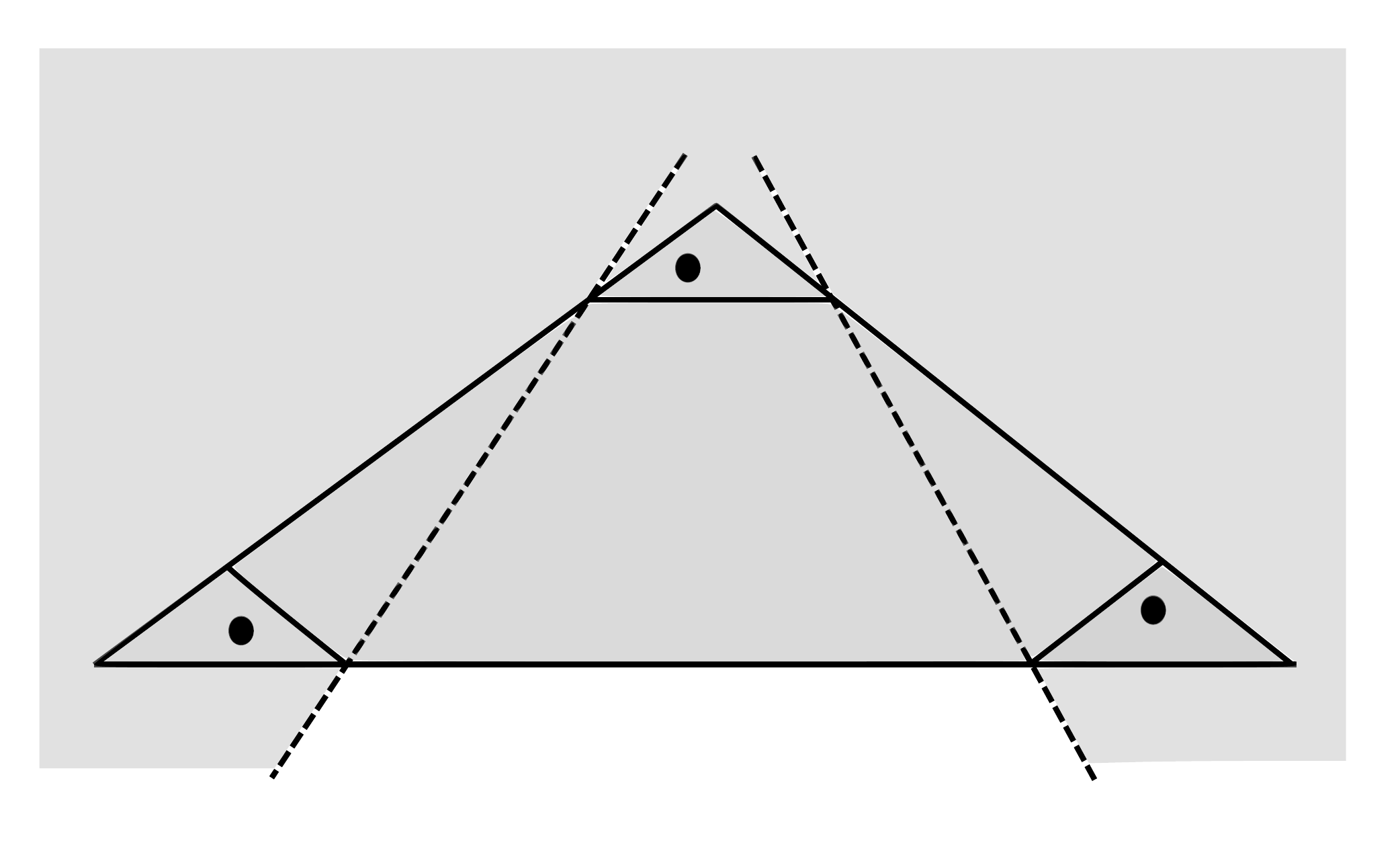} 
\put(113,92){\scriptsize $z_i^0$} 
\put(28,35){\scriptsize $z_i^1$} 
\put(185,34){\scriptsize $z_i^2$} 
\put(110,55){\scriptsize $\D_i$} 
\put(10,90){\scriptsize No points} 
\put(20,80){\scriptsize here}
\put(175,90){\scriptsize No points} 
\put(185,80){\scriptsize here}
\end{overpic}
\caption{The simplices $\D_i^j$, the points $z_i^j$ and the half-spaces $H_i^j$ (left). Illustration of the events $A_i$ (right).}
\label{fig:2}
\end{figure}

\begin{lemma}\label{lem:ProbabilityAi}
There exists a constant $c_3\in(0,1)$ such that $\bP(A_i)\geq c_3$ for all $i\in\{1,\ldots,m\}$.
\end{lemma}

The facts summarized so far have been used in \cite{BaranyVu} to prove a lower variance bound for the volume $V_d(K_n)$ of $K_n$. Since we are interested in all intrinsic volumes $V_1(K_n),\ldots,V_d(K_n)$, a refinement is necessary to obtain such bounds. In fact, we now follow and adapt the method already applied in \cite{BaranyFodorVigh,BorotzkyFodorReitznerVig,ReitznerCLT} to handle the more general situation.

\subsection{The effect of local perturbations}

Let $z\in\SSd$ and $G$ be a measurable subset of $\GG(d,\ell)$ for some $\ell\in\{0,\ldots,d\}$. The angle $\sphericalangle(z,G)$ between $z$ and $G$ is defined as $\min\{\sphericalangle(z,x):x\in L,L\in G\}$, where $\sphericalangle(z,x)=\arccos{\lan x,z\ran\over\|x\|}$ is the ordinary angle between $z$ and $x$. We observe the following geometric fact, see also \cite[Lemma 1]{BaranyFodorVigh}:

\begin{lemma}\label{lem:Angle}
Let $z\in\SSd$ and $\ell\in\{1,\ldots,d\}$. One can find a constant $c_4\in(0,\infty)$ only depending on $d$ and on $\ell$ such that  $$\nu_\ell(\{L\in\GG(d,\ell):\sphericalangle(z,L)\leq a\})\gg a^{d-\ell}$$
for all $0<a<c_4$.
\end{lemma}
\begin{proof}
For $M\in\GG(d,\ell-1)$ we denote by $\GG(M,\ell)$ the relative Grassmannian of $\ell$-dimensional linear subspaces of $\RRd$ containing $M$. This space is supplied with a unique Haar probability measure $\nu_\ell^M$, see Chapter 7.1 in \cite{SchWeil}. Similarly, we let $\GG(z^\perp,\ell-1)$ be the relative Grassmannian of $(\ell-1)$-dimensional linear subspaces of $\RRd$ that are contained in the hyperplane $z^\perp$ orthogonal to $1$-dimensional linear subspace spanned by $z$. The unique Haar probability measure on $\GG(z^\perp,\ell-1)$ is denoted by $\nu_{\ell-1}^{z^\perp}$. For $M\in\GG(z^\perp,\ell-1)$ let $u\in\SSd\cap M^\perp$ be such that $\sphericalangle(z,u)\leq a$. It is clear that the $\ell$-dimensional linear subspace ${\rm span}(M,u)$ spanned by $M$ and $u$ is contained in the set $\{L\in\GG(d,\ell):\sphericalangle(z,L)\leq a\}$ we are interested in. Formally, using Fubini's theorem for flag spaces (see \cite[Theorem 7.1.1]{SchWeil}) in the second step we write
\begin{align*}
&\nu_\ell(\{L\in\GG(d,\ell):\sphericalangle(z,L)\leq a\})\\
&\qquad= \int_{\GG(d,\ell)}{\bf 1}_{\{\sphericalangle(z,L)\leq a\}}\,\nu_\ell(\dint L)\\
&\qquad=\int_{\GG(d,\ell-1)}\int_{\GG(M,\ell)}{\bf 1}_{\{\sphericalangle(z,L)\leq a\}}\,\nu_{\ell}^M(\dint L)\nu_{\ell-1}(\dint M)\\
&\qquad\geq \int_{\GG(z^\perp,\ell-1)}\int_{\GG(M,\ell)}{\bf 1}_{\{\sphericalangle(z,L)\leq a\}}\,\nu_{\ell}^M(\dint L)\nu_{\ell-1}^{z^\perp}(\dint M)\\
&\qquad\geq\int_{\GG(z^\perp,\ell-1)}\int_{\SSd\cap M^\perp}{\bf 1}_{\{\sphericalangle(z,u)\leq a\}}\,\nu_{\SSd\cap M^\perp}(\dint u)\nu_{\ell-1}^{z^\perp}(\dint M)\\
&\qquad=\int_{\GG(z^\perp,\ell-1)}\nu_{\SSd\cap M^\perp}(\{u\in\SSd\cap M^\perp:\sphericalangle(z,u)\leq a \})\,\nu_{\ell-1}^{z^\perp}(\dint M)\,.
\end{align*}
Since $M^\perp$ has dimension $d-\ell+1$, the set of points $u\in\SSd\cap M^\perp$ with $\sphericalangle(z,u)\leq a$ forms a spherical cap in the $(d-\ell)$-dimensional subsphere $\SSd\cap M^\perp$ of $\SSd$. It has radius of order $a$ and volume of order $a^{d-\ell}$, where by volume we mean here the normalized $(d-\ell)$-dimensional Hausdorff measure $\nu_{\SSd\cap M^\perp}$ on $\SSd\cap M^\perp$. Hence, for sufficiently small $a$, we have
$$
\nu_{\SSd\cap M^\perp}(\{u\in\SSd\cap M^\perp:\sphericalangle(z,u)\leq a \}) \gg a^{d-\ell}
$$
and, since $\nu_{\ell-1}^{z^\perp}$ is a probability measure, also
$$\nu_\ell(\{L\in\GG(d,\ell):\sphericalangle(z,L)\leq a\})\gg a^{d-\ell}\,.$$
The proof is complete.
\end{proof}

For $i\in\{1,\ldots,m\}$ put $F_i:=[z_i^1,\ldots,z_i^d]$ and define
$$
\widetilde{V}_\ell(z;F_i) := {d\choose \ell}{\k_d\over\k_\ell\k_{d-\ell}} \int_{\GG(d,\ell)}{\bf 1}_{\{L\cap C_i^2\neq\emptyset\}}\,\vol_\ell([z,F_i]|L)\,\nu_\ell(\dint L)\,,\quad z\in\D_i^0\,.
$$
The next lemma provides a lower bound for the variance of these local functionals.

\begin{lemma}\label{lem:VarianceI}
Fix $\ell\in\{1,\ldots,d\}$, let $i\in\{1,\ldots,m\}$ and let $Z_i$ be a point chosen with respect to the normalized Gaussian measure restricted to $\D_i^0$. Then
$$
\bV_i[\widetilde{V}_\ell(Z_i;F_i)] \gg (\log n)^{-(d-\ell+1)}\,,
$$
where the notation $\bV_i[\,\cdot\,]$ refers to the variance that is taken with respect to $Z_i\in\D_i^0$. 
\end{lemma}
\begin{proof}
Denote by $w_i$ the centre of the facet of $\D_i^0$ opposite to the vertex $y_i^0$, and define the points $w_i^1:={2\over 3}y_i^0+{1\over 3}w_i$ and  $w_i^2:={1\over 3}y_i^0+{2\over 3}w_i$. Furthermore, the regions $R_i^1,R_i^2\subset\D_i^0$ are given by $R_i^1:=(w_i^1-C_i^2)\cap\D_i^0$ and $R_i^2:=(w_i^2+C_i^2)\cap\D_i^0$. It is crucial to observe that one can find a constant $c_7$ only depending on $d$ such that $V_d(R_i^k)\geq c_7V_d(\D_i^0)$ for $k=1$ and $k=2$. This follows from (\ref{eq:Ccones}). Together with the first part of Lemma \ref{lem:ContentDeltaIJ} and the fact that the Gaussian density \eqref{eq:DensityNormal} satisfies
$$
\varphi_d(x) \approx {\sqrt{\log n}\over n}\qquad\text{for all } x\in\RRd\text{ with }r\leq\|x\|\leq r+{1\over r}\,,
$$
we see that the Gaussian measure of $R_i^k$ is
\begin{equation}\label{eq:ContentRk}
\g_d(R_i^k)\approx n^{-1}\,,\qquad k\in\{1,2\}\,.
\end{equation} 

Next, fix some $L\in\GG(d,\ell)$ that intersects the interior of the polar of the cone $C_i^2$. This condition means that $L$ has an orthonormal basis  $e_1,\ldots,e_\ell$ such that  the hyperplane $H_{i,0}:=\{x\in\RRd:\lan x,e_1\ran= \lan w_i^2,e_1\ran\}$ has only one point (namely the origin) in common with $C_i^2$. Let $H_{i,0}^+$ be the half-space bounded by $H_{i,0}$ not containing the origin. Finally, let us define the set $G_i:=H_{i,0}^+\cap(w_i^1+C_i^2)\subset\D_i^0$. We choose points $Z_i^1\in R_i^1$ and $Z_i^2\in R_i^2$. The whole construction is illustrated in Figure \ref{fig4} where $L$ appears translated by $w_i^2$ (not affecting the $\vol_\ell(G_i|L)$).

Observe now that $R_i^1$ and $R_i^2$ are separated by the hyperplane $H_{i,0}$. Consequently, we have that $Z_i^2 \in [Z_i^1,F_i]$, which implies the inclusion $[Z_i^2,F_i]\subset[Z_i^1,F_i]$. In addition, $G_i \subset H_{i,0}^+$ and $R_i^2\cap H_{i,0}^+=\{w_i^2\}$, which yields $G_i\cap[Z_i^2,F_i]=\{w_i^2\}$. Finally, we observe that the hyperplane parallel to $H_{i,0}$ separates $R_i^1$ and $G_i$, whence $G_i\subseteq[Z_i^1,F_i]$. 
\begin{figure}[t]
\centering
\begin{overpic} 
[width=0.99\textwidth]{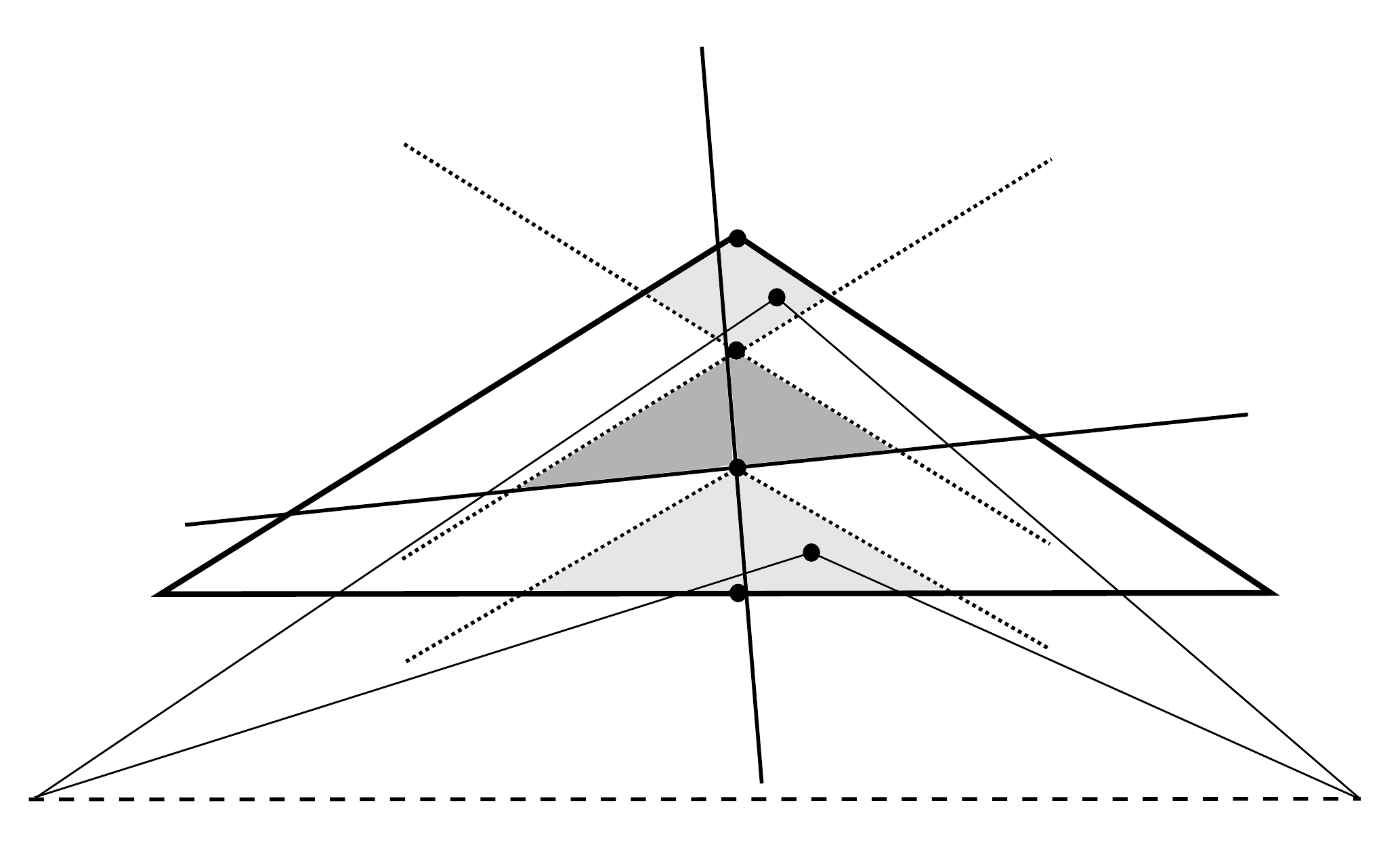} 
\put(0,8){\scriptsize $z_i^1$} 
\put(440,10){\scriptsize $z_i^2$} 
\put(210,0){\scriptsize $F_i=[z_i^1,z_i^2]$} 
\put(197,240){\scriptsize $L+w_i^2$} 
\put(228,72){\scriptsize $w_i$} 
\put(226,110){\scriptsize $w_i^2$} 
\put(243,157){\scriptsize $w_i^1$} 
\put(240,200){\scriptsize $y_i^0$} 
\put(236,178){\scriptsize $Z_i^1$} 
\put(250,100){\scriptsize $Z_i^2$} 
\put(200,90){\scriptsize $R_i^2$} 
\put(220,178){\scriptsize $R_i^1$} 
\put(210,130){\scriptsize $G_i$} 
\put(380,142){\scriptsize $H_{i,0}$} 
\put(413,78){\scriptsize $\D_i^0$} 
\end{overpic}
\caption{Construction in the proof of Lemma \ref{lem:VarianceI}.}
\label{fig4}
\end{figure}
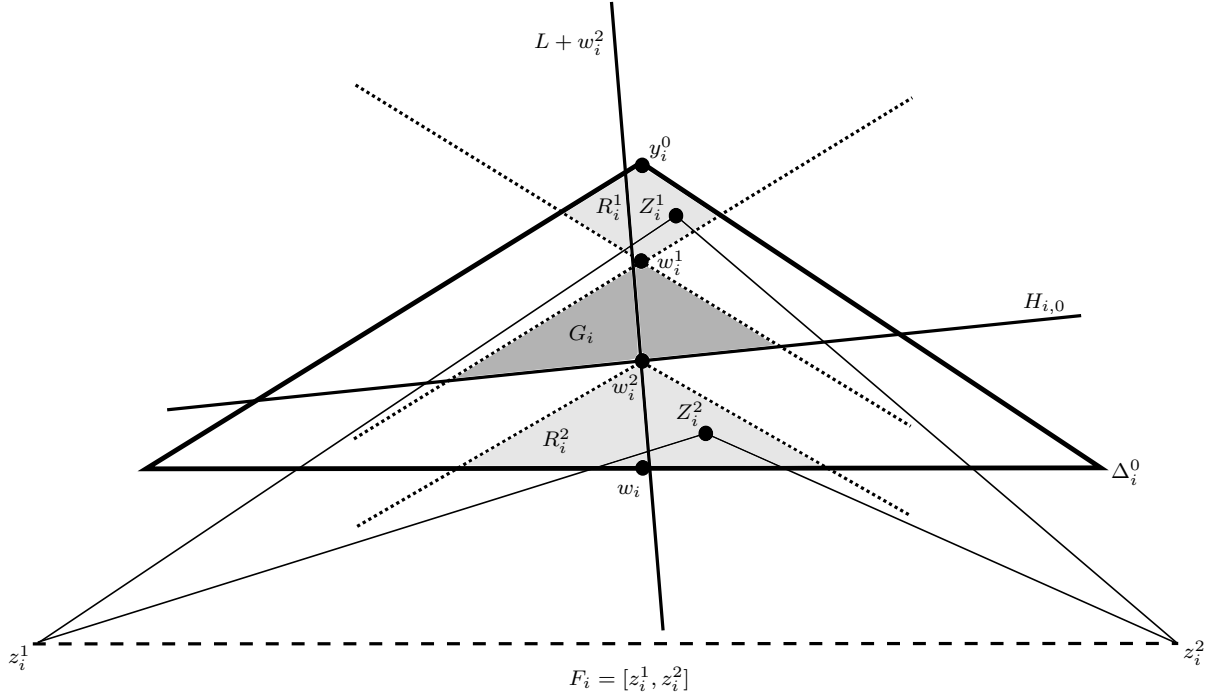
The construction also shows $V_d(G_i) \approx r^{-1}=(\log n)^{-{1\over 2}}$ and implies that
\begin{equation}\label{eq:LowerBoundGi}
\vol_\ell(G_i|L)\gg (\log n)^{-{1\over 2}}\,.
\end{equation}
As result, we arrive at the estimate
$$
\vol_\ell([Z_i^1,F_i]|L)-\vol_\ell([Z_i^2,F_i]|L) \geq \vol_\ell(G_i|L)\,.
$$
Hence,
\begin{align*}
& \widetilde{V}_\ell(Z_i^1;F_i)-\widetilde{V}_\ell(Z_i^2;F_i)\\
& = {d\choose \ell}{\k_d\over\k_\ell\k_{d-\ell}} \int_{\GG(d,\ell)}{\bf 1}_{\{L\cap C_i^2\neq\emptyset\}}\,\big(\vol_\ell([Z_i^1,F_i]|L)-\vol_\ell([Z_i^2,F_i]|L)\big)\,\nu_\ell(\dint L)\\
& \geq {d\choose \ell}{\k_d\over\k_\ell\k_{d-\ell}} \int_{\GG(d,\ell)}{\bf 1}_{\{L\cap C_i^2\neq\emptyset\}}\,\vol_\ell(G_i|L)\,\nu_\ell(\dint L)\\
&\gg  (\log n)^{-{1\over 2}}\,\nu_\ell(\{L\in\GG(d,\ell):L\cap C_i^2\neq\emptyset\})\\
&\gg (\log n)^{-{1\over 2}}\,(\log n)^{-{d-\ell\over 2}} \\
& = (\log n)^{-{d-\ell+1\over 2}}\,,
\end{align*}
where we used \eqref{eq:LowerBoundGi}, the definition of $C_i^2$, and Lemma \ref{lem:Angle}. Note that the latter can indeed be applied with $a=1/ \log n$, since $1/\log n <c_4$ for sufficiently large $n$. Selecting now $Z_i^k$, $k\in\{1,2\}$, independently at random according to the normalized Gaussian measure restricted to $\D_i^0$ (i.e., $Z_i^1$ and $Z_i^2$ are independent copies of $Z_i$), we conclude that
\begin{align*}
\bV[\widetilde{V}_\ell(Z_i;F_i)] &= {1\over 2}\bE\big[\big(\widetilde{V}_\ell(Z_i^1;F_i)-\widetilde{V}_\ell(Z_i^2;F_i)\big)^2\big]\\
&\geq {1\over 2}\bE\big[\big(\widetilde{V}_\ell(Z_i^1;F_i)-\widetilde{V}_\ell(Z_i^2;F_i)\big)^2\,{\bf 1}_{R_i^1}(Z_i^1){\bf 1}_{R_i^2}(Z_i^2)\big]\\
&\gg (\log n)^{-(d-\ell+1)}\,\bP(Z_i^1\in R_i^1,Z_i^2\in R_i^2)\,.
\end{align*}
To obtain a lower bound for $\bP(Z_i^1\in R_i^1,Z_i^2\in R_i^2)$ we recall \eqref{eq:ContentRk} and combine this with the second assertion of Lemma \ref{lem:ContentDeltaIJ} as well as with the independence of the random points $Z_i^1$ and $Z_i^2$. This implies that
$$
\bP(Z_i^1\in R_i^1,Z_i^2\in R_i^2)=\prod_{k=1}^2\bP(Z_i^k\in R_i^1) = \prod_{k=1}^2{\g_d(R_i^k)\over\g_d(\D_i^0)} \geq c_8^2
$$
with a constant $c_8\in(0,\infty)$ only depending on $d$. Hence,
$$
\bV[\widetilde{V}_\ell(Z;F_i)] \gg (\log n)^{-(d-\ell+1)}\,,
$$
completing thereby the proof of the lemma.
\end{proof}

\section{Proof of Theorem \ref{thm:Main}}\label{sec:ProofMain}

Recall the geometric construction and its properties from the previous section and denote by $\cF\subset\cA$ the $\sigma$-field generated by the random points $X_1,\ldots,X_n$, except those in the simplices $\D_i^0$ for which ${\bf 1}_{A_i}=1$, $i\in\{1,\ldots,m\}$. The conditional variance formula implies that
\begin{equation*}
\begin{split}
\bV[V_\ell(K_n)] &= \bE\big[\bV[V_\ell(K_n)|\cF]\big]+\bV\big[\bE[V_\ell(K_n)|\cF]\big]\\
&\geq\bE\big[\bV[V_\ell(K_n)|\cF]\big]\,.
\end{split}
\end{equation*}
Now, conditioned on $\cF$, suppose that ${\bf 1}_{A_i}=1$, write $Z_i$ for the (unique) random point in $\D_i^0$ and denote by $F_i$ the convex hull of the random points in $\D_i^j$ with $j\in\{1,\ldots,d\}$. We notice that if ${\bf 1}_{A_i}=1$ for each $i\in I$ in a subset $I\subset\{1,\ldots,m\}$, then $(\widetilde{V}_\ell(Z_i;F_i):i\in I)$ is a family of independent random variables as a consequence of the result of Lemma \ref{lem:ConesCi}. This independence property implies that
$$
\bV[V_\ell(K_n)|\cF] = \sum_{i=1\atop {\bf 1}_{A_i}=1}^m\bV_i[V_\ell(K_n)]=\sum_{i=1\atop {\bf 1}_{A_i}=1}^m\bV_i[\widetilde{V}_\ell(Z_i;F_i)]\,,
$$ 
where, as in the previous section, the notation $\bV_i[\,\cdot\,]$ refers to the variance that is taken only with respect to the point $Z_i\in\D_i^0$ and we only sum over those $i\in\{1,\ldots,m\}$ with the property that ${\bf 1}_{A_i}=1$. These variances can be controlled by means of Lemma \ref{lem:VarianceI}, which implies that
$$
\bV[V_\ell(K_n)|\cF] \gg (\log n)^{-(d-\ell+1)}\,\sum_{i=1}^m{\bf 1}_{A_i}\,.
$$
Taking expectations and finally applying Lemma \ref{lem:ProbabilityAi} as well as Lemma \ref{lem:EstimateM}, we arrive at
\begin{align*}
\bV[V_\ell(K_n)] &\gg (\log n)^{-(d-\ell+1)}\,\sum_{i=1}^m\bP(A_i) \\
&\gg (\log n)^{-(d-\ell+1)} \times (\log n)^{d-1\over 2} \\
&= (\log n)^{\ell-{d+3\over 2}}\,.
\end{align*}
This completes the argument and the proof of Theorem \ref{thm:Main}. \hfill $\Box$

\subsection*{Acknowledgement}

Both authors were partially supported by ERC Advanced Grant 267165 (DISCONV), and the first author by the Hungarian NKFIH grants 111827 and 116769. We are grateful to Julian Grote (Bochum) for insights concerning the concentration inequality for $V_\ell(\Pi_n)$. We would also like to thank the referee for his/her stimulating comments and suggestions.

\end{document}